\documentclass[12pt]{amsart}
\usepackage{amsfonts}
\usepackage{amsmath}
\usepackage{amssymb}
\usepackage{amsthm}
\usepackage[margin=1.2in]{geometry}
\usepackage{xcolor, hyperref} 
\usepackage[italic]{derivative}

\newtheorem{theorem}{Theorem}[section]

\newtheorem{lemma}[theorem]{Lemma}

\theoremstyle{definition}
\newtheorem{definition}[theorem]{Definition}
\newtheorem{remark}[theorem]{Remark}

\newcommand{\R}{\mathbb{R}}

\numberwithin{equation}{section}

\begin{document}
\title[Wiener-Ikehara theorem and Chebyshev bounds for Beurling primes]{A Wiener-Ikehara type theorem and its application to Chebyshev bounds for Beurling primes}

\author[Y.~Tranoy] {Yarne Tranoy}
\address{Y.~Tranoy\\  Department of Mathematics and Data Science \\ Vrije Universiteit
Brussel\\  Pleinlaan 2\\ 1050 Brussel\\ Belgium}
\email{yarne.tranoy@vub.be}

\author[J.~Vindas]{Jasson Vindas} 
\thanks{The work of J. Vindas was supported by Ghent University through the grant
bof/baf/4y/2024/01/155} 
\address{J.~Vindas \\Department of Mathematics: Analysis, Logic and Discrete Mathematics\\ Ghent University\\
  Krijgslaan 297\\ 9000 Ghent\\ Belgium} 
\email{jasson.vindas@UGent.be}


\begin{abstract} We provide a new version of the Wiener-Ikehara theorem where one deduces bounds
$$
0< \liminf_{x\to\infty} \frac{S(x)}{e^{x}}\leq \limsup_{x\to\infty} \frac{S(x)}{e^{x}} <\infty
$$
 for  (in particular) a non-decreasing function $S$ from a mild hypothesis on the boundary behavior of its Laplace transform on a vertical segment containing $s=1$. As an application, we establish new criteria for the validity of Chebyshev bounds for Beurling generalized prime number systems under weaker conditions than were known so far.

\end{abstract}

\maketitle

\begin{center}

\emph{Dedicated to J\'{a}nos Pintz on the occasion of his 75th birthday }

\end{center}

\bigskip

\section{Introduction}  \label{Sec Intro}
In its simplest form, the Wiener-Ikehara theorem states that if a non-decreasing function $S$ on $[0,\infty)$ has convergent Laplace transform $\int_{0}^{\infty} e^{-sx}S(x)\mathrm{d}x$ on the half-plane $\Re e \:s>1$ and if there is a constant $a\in \mathbb{R}$ such that \begin{equation}
\label{eq: Laplace extension W-I}
G(s)= \mathcal{L}\{S; s\}-\frac{a}{s-1}
\end{equation}
admits analytic extension across $\Re e\: s=1$, then $S$ has asymptotic behavior
\begin{equation}
\label{eq: asymp S W-I}
S(x)\sim a e^{x}, \qquad x\to\infty.
\end{equation}
This celebrated complex Tauberian theorem has been extensively studied over the last century and has been generalized in many ways; see the monograph  \cite[Chapter III]{korevaarbook} for a classical account on the Wiener-Ikehara theorem and  \cite{B-D-V2021,C-V2026,D-V2016,korevaar2005,revesz-roton,zhang2014} for more recent contributions. 

The most general versions of the Wiener-Ikehara theorem leading to \eqref{eq: asymp S W-I} involve local pseudofunction boundary behavior. The pseudofunction approach was initiated by Korevaar in \cite{korevaar2005} in order to take the boundary requirements on the Laplace transform to a minimum; see Section \ref{sect preliminaries} below for precise definitions and some background material on local pseudofunctions (and psedomeasures). The following exact form of the Wiener-Ikehara theorem from \cite{D-V2016} supplies a complete characterization of the asymptotic formula \eqref{eq: asymp S W-I}. In addition to local pseudofunction boundary behavior, it makes use of log-linear slow decrease as Tauberian hypothesis. A function $S$ is called \emph{log-linearly slowly decreasing} (at $\infty$) if for each $\varepsilon>0$ one can find $\delta,x_0>0$ such that
\[
\frac{S(y) - S(x)}{e^{x}}\geq -\varepsilon \qquad \mbox{for } x \leq y \leq x+\delta \mbox{ and } x\geq x_{0}.
\]

\begin{theorem}[{Exact Wiener-Ikehara theorem {\cite[Theorem 3.6]{D-V2016}}}] \label{th exact W-I} Let $S\in L^{1}_{loc}[0,\infty)$. Then, the asymptotic relation \eqref{eq: asymp S W-I} holds if and only if $S$ is log-linearly slowly decreasing, its Laplace transform is  convergent for $\Re e\: s>1$, and the analytic function $G$ given by \eqref{eq: Laplace extension W-I} has local pseudofunction boundary behavior on the line $\Re e \:s=1$.
\end{theorem}

The first aim of this paper is to establish a new Wiener-Ikehara type theorem that yields lower and upper bounds
\begin{equation}
\label{eq: l-u bounds W-I} e^{x}\ll S(x)\ll e^{x},
\end{equation}
instead of the classical stronger conclusion \eqref{eq: asymp S W-I}. (Here $\ll$ stands for Vinogradov's symbol, that is, $f(x)\ll g(x)$ means that $f(x)=O(g(x))$, $x\to\infty$.) As observed in \cite[p.~581]{D-V2018}, a small adjustment in the Graham-Vaaler finite form of the Wiener-Ikehara theorem \cite[Theorem 10, p.~294]{G-V1981} proves that \eqref{eq: l-u bounds W-I} holds whenever $S$ is non-decreasing and its Laplace transform has local pseudofunction boundary behavior on an open boundary line segment containing the point $s=1$. The main feature of the Tauberian theorem we shall obtain in this article is that it is capable of delivering \eqref{eq: l-u bounds W-I} under much milder boundary requirements on the Laplace transform than local pseudofunction behavior at $s=1$. 

Our novel version of the Wiener-Ikehara theorem (Theorem \ref{the Tth}) is stated and proved in Section \ref{sect Tauberian theorem}. It is inspired by both Diamond and Zhang's \cite[Chapter 11]{diamond-zhangbook} and Vindas' \cite{VindasChebyshevI,VindasChebyshevII}  analytic approaches to Chebyshev bounds for Beurling generalized prime number systems. 
In fact, the second aim of this work is to simplify the treatment of Chebyshev bounds through an explicit Tauberian perspective. Naturally, as it is the case with most Tauberian theorems, Theorem \ref{the Tth} might be expected to have potential applications beyond analytic number theory. 

Let $N$ and $\pi$ be the integer and prime counting functions of a Beurling generalized number system \cite{beurling1937,diamond-zhangbook}. We shall show the following criterion for the validity of Chebyshev bounds:

\begin{theorem}\label{Th Chebyshev bounds} Suppose that 
\begin{align}
\label{eq N intro}
N(x)= x &\log^{r}x \ \bigg(a  +\sum_{\nu=1}^{m}a_\nu\cos(\beta_\nu+\theta_{\nu} \log x)\bigg)  \\
 \nonumber
&+x\sum_{j=1}^{k}\log ^{r_j }x  \sum_{\nu=1}^{m_j}P_{\nu,j}(\log \log x)\cos(\beta_{\nu,j}+\theta_{\nu,j} \log x)   +E(x), \quad x>e,
\end{align}
 where 
\begin{equation}
\label{Ebound1}
a>0, \qquad r>-1, \qquad r>r_1>\dots> r_{k}, \qquad  \theta_{\nu}\neq0, 
\end{equation}
the $P_{\nu,j}$ are polynomials, and the function $E$ satisfies
\begin{equation}
\label{Ebound2}
    \int_e^\infty \frac{|E(u)|}{u^{2}} \mathrm{d}u<\infty
\end{equation}
and
\begin{equation}
\label{Ebound3}
    \int_e^x  \frac{E(u)\log u}{u}\: \mathrm{d}u \ll x.
\end{equation}
Then, the generalized primes satisfy the Chebyshev bounds, that is,
\begin{equation}\label{eq Chebyshev bounds}
\frac{x}{\log x}\ll \pi(x)\ll \frac{x}{\log x}\: .
\end{equation}
\end{theorem} 

Observe that a simpler condition yielding both \eqref{Ebound2} and \eqref{Ebound3} is $E(x)= O(x/\log^{\alpha}x)$ with $\alpha>1$. Naturally, the integral average hypothesis \eqref{Ebound3} is satisfied if the stronger pointwise bound $E(x)=O(x/\log x)$ holds.  Theorem \ref{Th Chebyshev bounds} generalizes \cite[Theorem 11.1]{diamond-zhangbook}. Furthermore, it contains as particular instances other criteria for Chebyshev bounds available so far in the literature (\cite{beurling1937,diamond3,diamond-zhang13a,kahane98,Revesz94,VindasChebyshevI,VindasChebyshevII, zhang1993,zhang2014}).
 
 We will actually deduce Theorem \ref{Th Chebyshev bounds} in Section \ref{sect Chebyshev bounds} from a still more general result (Theorem \ref{th CB}), which gives sufficient conditions for \eqref{eq Chebyshev bounds} (or equivalently \eqref{eq: CB}) in terms of the boundary behavior of the Beurling zeta function of a generalized number system and its derivative near the point $s=1$.

\section{Preliminaries}\label{sect preliminaries}

This section collects several auxiliary notions and lemmas.

\subsection{Distributions and Fourier transform}
We shall make use of standard Schwartz distribution calculus in our manipulations. Our notation for generalized functions is the same as in \cite{estrada-kanwal,vladimirovbook}.

The standard Schwartz test function spaces of compactly supported smooth functions (on an open subset $I\subseteq \mathbb{R}$)  and rapidly decreasing smooth functions are denoted by $\mathcal{D}(I)$ and $\mathcal{S}(\mathbb{R})$, while $\mathcal{D}'(I)$ and $\mathcal{S}'(\mathbb{R})$ stand for their topological duals, the spaces of distributions and tempered distributions.  We write $\langle f,\varphi\rangle$, or $\langle f(x),\varphi(x)\rangle$ with the use of a dummy variable of evaluation, for the dual pairing between a distribution $f$ and a test function $\varphi$. Accordingly, sometimes we also write expressions like $f(x)\in \mathcal{D}'(\R)$, to be interpreted as $f\in \mathcal{D}'(\R)$. As usual, locally integrable functions are regarded as (regular) distributions via $\langle f(x),\varphi(x)\rangle=\int_{-\infty}^{\infty}f(x)\varphi(x)\mathrm{d}x$.  

We normalize Fourier transforms as $\widehat{\varphi}(t)=\mathcal{F}\{\varphi;t\}=\int_{-\infty}^{\infty}e^{-itx}\varphi(x)\:\mathrm{d}x$, and  interpret them in the sense of tempered distributions when the integral definition does not make sense.
So, if $f\in\mathcal{S}'(\mathbb{R})$, its Fourier transform is the tempered distribution $\widehat{f}\in\mathcal{S}'(\mathbb{R})$ determined by $\langle \widehat{f}(t),\varphi(t)\rangle=\langle f(x), \widehat\varphi(x)\rangle$ for all test functions $\varphi\in \mathcal{S}(\mathbb{R})$. 

\subsection{Pseudofunctions, pseudomeasures, and boundary values}

We denote as $A(\mathbb{R})=\mathcal{F}(L^{1}(\mathbb{R}))$ the global Wiener algebra, its dual $PM(\mathbb{R})=\mathcal{F}(L^{\infty}(\mathbb{R}))$ is the space of global pseudomeasures. We call $f\in PM(\mathbb{R})$ a global pseudofunction if additionally ($\operatorname*{ess}$) $\lim_{|x|\to\infty} \widehat{f}(x)=0$, and write $f\in PF(\mathbb{R})$. Since the Fourier transform sends convolution into multiplication, $A(\mathbb{R})$ is an algebra of continuous functions under pointwise multiplication, while $PM(\mathbb{R})$ and $ PF(\mathbb{R})$ have natural multiplicative $A(\mathbb{R})$-module structures. 

Let $I\subset \mathbb{R}$ be open. Consider $t_0\in I$ and let $X$ stand for $A,PF,$ or $PM$ (or similarly other suitable spaces). 
A Schwartz distribution $g\in\mathcal{D}'(I)$ is said to be an element of $X_{loc}(t_0)$ if there is a neighborhood of $t_0$ in $I$ where $g$ coincides with an element of the global space $X(\mathbb{R})$. We then set 
$$
X_{loc}(I)=\{g\in \mathcal{D}'(I): g\in X_{loc}(t_0),\: \forall t_0\in I\}.
$$
Clearly, $X_{loc}(\cdot)$ is a sheaf. (In fact, a fine sheaf, which can be established exactly as in \cite[``Theorem of piecewise sewing", p.~14]{vladimirovbook} due to the existence of $C^{\infty}$-partitions of unity for instance.)
We term $A_{loc}(I)$ the local Wiener algebra on $I$, and $PF_{loc}(I)$ and $PM_{loc}(I)$ the spaces of local pseudofunctions and pseudomeasures on $I$. One can readily check that local pseudofunctions are characterized by a generalized Riemann-Lebesgue lemma: 

\begin{lemma}[{\cite[Lemma 3.3 and Lemma 3.4]{korevaar2005}}] A distribution
 $g\in \mathcal{D}'(I)$ is a local pseudofunction on $I$ if and only if for each $\varphi\in\mathcal{D}(I)$
\[
\left\langle g(t),e^{iht}\varphi(t)\right\rangle=o(1), \quad |h|\to\infty .
\]
\end{lemma}

Thus any locally Lebesgue integrable function is an example of a pseudofunction (thanks to the classical Riemann-Lebesgue lemma). 
We have in fact the (strict) inclusions 
\[
C^{\infty}(I)\subset A_{loc}(I)\subset L^{1}_{loc}(I)\subset PF_{loc}(I)\subset PM_{loc}(I). 
\]

It is important to point out that the multiplication on the global spaces induces an $A_{loc}$-module structure on the corresponding local counterparts. We collect this fact in the following lemma.

\begin{lemma}
\label{lemma mult}
If  $g_1\in A_{loc}(I)$ and $g_2\in X_{loc}(I)$, where the space $X$ stands for $PM, PF,$ or $A$, then $g_1 \cdot g_2\in X_{loc}(I)$.
\end{lemma}

Let $G(s)$ be analytic on the half-plane $\Re e\:s>1$ and let $I\subset \mathbb{R}$ be open. We say that $G$ has distributional boundary values on $1+iI$ if $G$ admits a boundary distribution $g\in\mathcal{D}'(I)$, that is, if 
\begin{equation*}
\lim_{\sigma\to1^{+}}\int_{-\infty}^{\infty}G(\sigma+it)\varphi(t)\mathrm{d}t=\left\langle g(t),\varphi(t)\right\rangle, \qquad \mbox{for each } \varphi\in\mathcal{D}(I).
\end{equation*}
We often write $g(t)=G(1+it)$ for its boundary value distribution. One then says that $G$ has local pseudomeasure boundary behavior at a point $s=1+it_0$ if $g\in PM_{loc}(t_0) $. Likewise, one defines boundary behavior with respect to other local spaces. We emphasize that $L^1_{loc}$-boundary behavior, continuous, or analytic extension are very special cases of local pseudofunction boundary behavior.

\subsection{Two basic lemmas from Wiener's theory} 
The next lemma is very simple but yet lies at the heart of various approaches to Wiener's Tauberian theorem. A proof can be found in \cite{korevaar1965}, but the reader can also easily verify it by himself.

\begin{lemma}[{\cite[Lemma 1]{korevaar1965}}]
\label{lemma korevaar}
Let $f,\phi\in L^{1}(\mathbb{R}) $. Set 
\begin{equation}
\label{mollifier eq}
\phi_{n}(x)=\frac{1}{n}\phi\left(\frac{x}{n}\right), \qquad n=1,2,\dots . 
\end{equation}
Then
\[
\lim_{n\to\infty} \left\|f\ast \phi_n - \widehat{f}(0) \phi_n\right\|_{L^{1}(\mathbb{R})}=0.
\]
\end{lemma}

 Let us recall Wiener's local division lemma, which we rephrase as follows.

\begin{lemma}[{Wiener's division lemma \cite[Theorem 7.3, p.~81]{korevaarbook}}]\label{Wiener lemma} Let $f\in A_{loc}(0)$. If $f(0)\neq 0$, then $1/f \in A_{loc}(0)$.
\end{lemma}

\subsection{Controlled decrease: Tauberian conditions}

The controlled decrease notions from the next definition will act as Tauberian conditions in our considerations. 
\begin{definition}
    Let $S$ be a real valued function on $[0,\infty)$. It is called:
    \begin{enumerate}
    \item [(i)] \emph{log-linearly boundedly decreasing} if there are $\delta,x_0, M>0$ such that
    \[\frac{S(y)-S(x)}{e^{x}}\geq - M  \qquad \mbox{for } x \leq y \leq x+\delta \mbox{ and } x\geq x_{0} \:   ;
    \]
    \item [(ii)] \emph{strongly log-linearly slowly decreasing} if there is a non-increasing positive function $\eta$ satisfying $\eta(x)=o(1)$ as $x\to\infty$ and there are constants $\delta,x_0>0$ such that
      \begin{equation}\label{eq: slsd}
\frac{S(y) - S(x)}{e^{x}}\geq -\eta(x) \qquad \mbox{for } x \leq y \leq x+\delta \mbox{ and } x\geq x_{0}.
\end{equation}
    \end{enumerate}
  \end{definition}

Despite the wording, any non-decreasing function is the simplest example of a strongly log-linearly slowly decreasing function. Note that the log-linear slow decrease defined at the Introduction lies  between log-linear bounded decrease and strong log-linear slow decrease. We mention that log-linear bounded decrease was introduced in \cite{D-V2016} and plays an important role in the Wiener-Ikehara type characterization of the bound $S(x)=O(e^{x})$, which we state in the next lemma.

\begin{lemma}[{\cite[Proposition 3.1]{D-V2016}}] \label{thboundedness} Let $S\in L^{1}_{loc}[0,\infty)$. Then, 
 \begin{equation}
 \label{W-Ieq1}
 S(x) =O(e^{x}), \quad x\to\infty,
\end{equation}
if and only if $S$ is log-linearly boundedly decreasing, its Laplace transform converges for $\Re e \: s > 1$, and $
 \mathcal{L}\{S;s\} $  admits pseudomeasure boundary behavior at the point $s = 1$.
\end{lemma}

We shall also need the following iterated version of the inequality \eqref{eq: slsd}.
\begin{lemma}\label{lemma log-linear slow decrease} Let $S$ be strongly log-linearly slowly decreasing. There is a non-increasing positive function  $\eta(x)=o(1)$ and $x_0>0$ such that 
\begin{equation}
\label{eq: s slsd}
S(x+h)-S(x)\geq -\eta(x) e^{x}(1+he^{h} ) \qquad \mbox{for any } x\geq x_0 \mbox{ and }h\geq0.  
\end{equation}
\end{lemma} 
\begin{proof}
Let $\delta$, $x_0$, and $\eta$ be as in \eqref{eq: slsd}. We may additionally assume that $\delta<1$. Find a non-negative integer $n$ such that $n\delta\leq h<(n+1)\delta$. Iterating \eqref{eq: slsd} and using that $\eta$ is non-increasing, we obtain
\[
S(x+h)  \geq S(x)-\eta(x) e^{x} - \eta(x)e^{x+h} \sum_{k=1}^{n}e^{-\delta k}\geq S(x)-\delta^{-1}\eta(x)e^{x}(1+ he^{h}). \]
So, the inequality \eqref{eq: s slsd} holds upon renaming $\delta^{-1}\eta(x)$ as $\eta(x)$.
\end{proof}

\section{The Tauberian theorem}\label{sect Tauberian theorem}
We are ready to state and prove our new Wiener-Ikehara type Tauberian theorem. 

\begin{theorem}\label{the Tth} Let $S\in L^{1}_{loc}[0,\infty)$ be strongly log-linearly slowly decreasing. Suppose that $\mathcal{L}\{S; s\}$ is convergent for $\Re e\: s>1$ and assume that for some $a>0$ the analytic function $G$ defined as \eqref{eq: Laplace extension W-I} admits distributional boundary value $g\in \mathcal{D}'(I)$ on the boundary line segment $1+iI$ for some real bounded open interval $I$ containing 0, that is,
  \begin{equation}
  \label{eq: taub 1}
  g(t)=G(1+it)=\lim_{\sigma\to1^+}G(\sigma+it)\qquad \text{in } \mathcal{D}'(I).
  \end{equation}  
If we can write $g=g_1\cdot g_2 +g_3$, where
    \begin{equation}
  \label{eq: taub 2}
    g_1\in A_{loc}(0) \mbox{ and }g_1(0)=0, \qquad g_2\in PM_{loc}(0), \qquad \mbox{and }  g_3\in PF_{loc}(0),
    \end{equation}
then $S$ satisfies the upper and lower bounds
\begin{equation}
\label{eq2: l-u bounds W-I}
e^{x}\ll S(x) \ll e^{x}, \qquad x\to\infty.
\end{equation}

\end{theorem}

\begin{proof} In what follows, we write $H$ for the Heaviside function, that is, the characteristic function of the interval $[0,\infty)$. Clearly, $\widehat{H}\in PM(\mathbb{R})=\mathcal{F}(L^{\infty}(\mathbb{R}))$. In view of Lemma \ref{lemma mult}, we obtain that $g\in PM_{loc}(0)$, so that $g+a\widehat{H}\in PM_{loc}(0)$. We then have
\begin{align*}
\lim_{\sigma\to 0^{+}} \mathcal{L}\{S; 1+\sigma+it\}&= g(t) + \lim_{\sigma\to0^{+}}\frac{a}{\sigma+it}\\
&= g(t) + a \lim_{\sigma\to0^{+}}\mathcal{L}\{H; \sigma+it\}
\\
&= g(t)+a  \widehat{H}(t) \qquad \mbox{in } \mathcal{D}'(I),
\end{align*}
 where in the last line we have used the well-known fact that the Fourier transform of a tempered distribution with support in $[0,\infty)$ coincides with the distributional boundary value of its Laplace transform on the imaginary axis \cite[Example~6.6.9, p.~100]{vladimirovbook}.
An application of Lemma \ref{thboundedness}  yields $S(x)=O(e^{x})$. Find $c>0$ such that
\begin{equation}
\label{proof Tt eq1}
S(x)\leq ce^{x}
\end{equation}
for  all large enough $x$. Since modifying $S$ on a finite interval does not change our hypothesis on the boundary behavior of its Laplace transform (the Laplace transform of a compactly supported function is an entire function \cite{vladimirovbook} and therefore a local pseudofunction on $1+i\mathbb{R}$), we may assume without loss of generality that \eqref{proof Tt eq1} holds for all $x>0$. 

It remains to establish the lower bound for $S(x)$. Replacing $I$ by a smaller interval if necessary, we may assume that $g_1\in A_{loc}(I)$, $g_2\in PM_{loc}(I),$ and  $g_3\in PF_{loc}(I).$

Choose a non-negative even test function $\phi\in\mathcal{S}(\mathbb{R})$ such that $\int_{-\infty}^{\infty}\phi(x)\mathrm{d}x=1$ and $\operatorname*{supp} \widehat{\phi}\subset I$. Furthermore, we consider the sequence of functions  $\phi_n$ given by \eqref{mollifier eq}. Their Fourier transforms $\widehat{\phi}_{n}(x)= \widehat{\phi}(nx)$ also satisfy $\operatorname*{supp} \widehat{\phi}_{n}\subset I$ and $\int_{-\infty}^{\infty}\phi_n(x)\mathrm{d}x=1$. Let $\varphi_n\in\mathcal{D}(I)$ be such that $\phi_n=\widehat{\varphi}_n$. 

We extend $S$ to be 0 on $(-\infty,0]$. We can regard $e^{-x}S(x)=O(1)$ as an element of $\mathcal{S}'(\mathbb{R})$. Using again that the distributional boundary value on the imaginary axis of a Laplace transform is  the Fourier transform, we conclude that the Fourier transform of $e^{-x}S(x)-aH(x)$ is given by $g$ on the interval $I$.
Hence, as $y\to\infty$,
    \begin{align*}
        \int_0^\infty e^{-x}S(x)\phi_n (x-y)\mathrm{d}x&= a \int_{-y}^\infty\phi_n (x)\mathrm{d}x+\langle g(t), e^{iyt}\varphi_n(t)\rangle\\
      &= a +o_{n}(1) + \langle g_3(t), e^{iyt}\varphi_n(t)\rangle + \langle g_1(t)\cdot g_2(t), e^{iyt}\varphi_n(t)\rangle, 
    \end{align*}
where we use the subscript $n$ to emphasize that by \emph{no means} the little $o$ term is meant to hold uniformly in $n$. By hypothesis, we also have $\langle g_3(t), e^{iyt}\varphi_n(t)\rangle=o_n(1)$. Summarizing,
\begin{equation}
\label{proof Tt eq2}
   \int_0^\infty e^{-x}S(x)\phi_n (x-y)\mathrm{d}x
    =a +o_{n}(1) + \langle g_1(t)\cdot g_2(t), e^{iyt}\varphi_n(t)\rangle, \quad y\to\infty.
\end{equation}

We must analyze the term $\langle g_1(t)\cdot g_2(t), e^{iyt}\varphi_n(t)\rangle$. Let $f_1\in L^{1}(\mathbb{R})$ and $f_{2}\in L^{\infty}(\mathbb{R})$ be such that $g_1=\widehat{f}_{1}$ and $g_2=\widehat{f}_{2}$ on an open interval containing $\operatorname*{supp}{\varphi_1}=\operatorname*{supp} \widehat{\phi}$. Since $\widehat{f}_{1}(0)=0$, Lemma \ref{lemma korevaar} yields
\begin{align*}
\lim_{n\to\infty}\sup_{y\in\R} |\langle g_1(t)\cdot g_2(t), e^{iyt}\varphi_n(t)\rangle|&=\lim_{n\to\infty}\left\| f_2\ast f_1 \ast \phi_n \right\|_{L^{\infty}(\R)}\\
&\leq \left\| f_2 \right\|_{L^{\infty}(\R)} \lim_{n\to\infty}\left\| f_1 \ast \phi_n \right\|_{L^{1}(\R)}=0.    
\end{align*}

Throughout the remaining of the proof, we pick $n$ large enough but fixed such that $\sup_{y\in\R} |\langle g_1(t)\cdot g_2(t), e^{iyt}\varphi_n(t)\rangle|\leq a/4$. Inserting this bound back into \eqref{proof Tt eq2} and choosing $y_1>0$ such that the term $o_n(1)$ in \eqref{proof Tt eq2} satisfies $|o_n(1)|\leq a/4$ for all $y\geq y_1$, we conclude that 
\begin{equation}
\label{proof Tt eq3}
   \int_0^\infty e^{-x}S(x)\phi_n (x-y)\mathrm{d}x\geq \frac{a}{2}, \qquad \mbox{for all }y\geq y_1.
\end{equation}

 The final step is to combine the Tauberian condition of strong log-linear slow decrease with the upper bound \eqref{proof Tt eq1} and the estimate \eqref{proof Tt eq3} in order to extract the desired lower bound in \eqref{eq2: l-u bounds W-I}. We first select a fixed $u>0$ such that 
 \begin{equation}
 \label{proof Tt eq4}
 \int_{|x|>u}\phi_n(x)\mathrm{d}x< \frac{a}{4c}.
 \end{equation}
 Let also $x_0>0$ and $\eta(x)=o(1)$ be as in Lemma \ref{lemma log-linear slow decrease} so that \eqref{eq: s slsd} holds. Keeping $y>\max\{y_1,u,x_0+u\}$ and employing \eqref{proof Tt eq1}, \eqref{proof Tt eq3}, \eqref{proof Tt eq4}, and \eqref{eq: s slsd}, we deduce that
     \begin{align*}
     \frac{a}{2}&\leq\int_0^\infty e^{-x}S(x)\phi_n(x-y)\mathrm{d}x
     \\
     &\leq c\int_{|x-y|\geq u}\phi_n(x-y)\mathrm{d}x + \int_{y-u}^{y+u} e^{-x}S(x)\phi_n(x-y)\mathrm{d}x
     \\
     & \leq  \frac{a}{4} + \left(e^{2u} \frac{S(y+u)}{e^{y+u}} + \eta(y-u)(1+2ue^{2u})\right)\int_{-u}^{u}\phi_{n}(x)\mathrm{d}x
     \\&\leq   \frac{a}{4} + e^{2u} \frac{S(y+u)}{e^{y+u}} + \eta(y-u)(1+2ue^{2u}).
     \end{align*}
Consequently,
\[\liminf_{x\to\infty} \frac{S(x)}{e^{x}}= \liminf_{y\to\infty} \frac{S(y+u)}{e^{y+u}} \geq \frac{a}{4e^{2u}}- (2u+e^{-2u}) \limsup_{y\to\infty}\eta(y-u)= \frac{a}{4e^{2u}}>0,\]
which completes the proof of \eqref{eq2: l-u bounds W-I}.
\end{proof}

\section{Application: Chebyshev bounds for Beurling primes}\label{sect Chebyshev bounds}
We shall now employ Theorem \ref{the Tth} to study Chebyshev bounds for Beurling primes. Theorem \ref{th CB} below provides sufficient conditions for \eqref{eq Chebyshev bounds} in terms of the Beurling zeta function of the number system,
\[ \zeta(s)= \int_{1^{-}}^{\infty}x^{-s} \mathrm{d}N(x).
\]

 As in classical analytic number theory, one verifies via elementary manipulations that the Chebyshev bounds \eqref{eq Chebyshev bounds} are equivalent to 
 \begin{equation}
 \label{eq: CB}
 x\ll \psi(x)\ll x.
 \end{equation} 
Here $\psi(x)$ stands for the Chebyshev function of the number system \cite{diamond-zhangbook}. We recall its Mellin-Stieltjes transform relates to the Beurling zeta function via the identity 
\begin{equation}\label{Chebyshev formula}
\int_{1}^{\infty} x^{-s}\mathrm{d}\psi(x)= -\frac{\zeta'(s)}{\zeta(s)}.
\end{equation}

\begin{theorem}
\label{th CB} Let $\zeta(s)$ be convergent for $\Re e\:s>1$. Suppose 
\[\zeta(s)=\frac{c}{(s-1)^{\rho}}+\sum_{j=1}^{k}\frac{P_{j}(\log(s-1))}{(s-1)^{\rho_{j}}} + F(s),\]
where 
\[ 
c>0, \qquad \rho>0, \qquad \rho>\rho_1>\dots> \rho_{k}, 
\]
 the $P_{j}$ are polynomials,
and the analytic function $F(s)$ has a continuous extension to some open boundary line segment containing $s=1$. If $F(1+it)\in A_{loc}(0)$ and\footnote{A word about the notation. Here $F'(1+it)$ stands for the distributional boundary value of $F'(s)$, which coincides with \[F'(1+it)= -i \odv{}{t} F(1+it) .\] The latter derivative should always be understood in the sense of Schwartz distributions \cite{estrada-kanwal,vladimirovbook}, as the continuous function $F(1+it)$ might not be differentiable in the classical sense and $F'(1+it)$ might thus not be a classical function.} $F'(1+it)\in PM(0)$, then  the Chebyshev bounds \eqref{eq: CB} hold true.
\end{theorem}
\begin{proof}
We apply Theorem \ref{the Tth} to the non-decreasing function $S(x)=\psi(e^{x})$ with constant given by $a=\rho$. We have to study the boundary behavior of the analytic function \eqref{eq: Laplace extension W-I}, which, in view of \eqref{Chebyshev formula}, is given by
\begin{align*}
    G(s)&=\mathcal{L}\{\psi(e^u), s\}-\frac{\rho}{s-1}\\
        &=-\frac{\zeta'(s)}{s\zeta(s)}-\frac{\rho}{s-1}\\
        &=-\frac{\rho}{s}-\frac{1}{s}\odv{}{s} [\log \left((s-1)^{\rho}\zeta(s)\right)]
        \\
        &
       = -\frac{\rho}{s} - \frac{1}{s} \cdot \frac{B'(s)}{B(s)},
\end{align*}
with $B(s)=(s-1)^{\rho}\zeta(s)$. Since the boundary value of $1/s$ on $1+i\mathbb{R}$ is the non-vanishing smooth function $(1+it)^{-1}$, and smooth functions are both themselves local pseudofunctions and multipliers for each of the three spaces $A_{loc}$, $PM_{loc},$ and $PF_{loc}$, we see that $G(s)$ satisfies the hypotheses of Theorem \ref{the Tth} if and only if $B'(s)/B(s)$ does it. It therefore suffices to show that the distributional boundary value of the latter analytic function can be written as
\begin{equation}\label{decompose B}
\frac{B'(1+it)}{B(1+it)}= g_{1}(t)\cdot g_{2}(t)+g_{3}(t), 
\end{equation}
in some ($t$-)neighborhood of 0, where $g_1,$ $g_2,$ and $g_3$ satisfy \eqref{eq: taub 2}. 

We write
\[
B(s)=D(s)+(s-1)^{\rho} F(s),
\]
where
\[
D(s)= c+\sum_{j=1}^{k}P_{j}(\log (s-1))(s-1)^{\rho-\rho_{j}},
\]
so that
\begin{equation}\label{decompose B 2}
\frac{B'(s)}{B(s)}= \frac{D'(s)+ \rho(s-1)^{\rho-1}F(s)+ (s-1)^{\rho}F'(s)}{B(s)}.
\end{equation}
The distribution $B(1+it)\in\mathcal{D}'(\mathbb{R})$ is continuous near $t=0$ and $B(1)=c\neq 0$. 

Note that $s^{\alpha}$ is the Laplace transform of $x_{+}^{-\alpha-1}/\Gamma(-\alpha)$. The tempered distributions $x_{+}^{-\alpha-1}/\Gamma(-\alpha)$ are discussed in detail for instance in \cite[Section 2.4]{estrada-kanwal}, where it is essentially shown\footnote{One has $x_{+}^{-\alpha-1}= x^{-\alpha-1}H(x)\in L^{1}_{loc}(\mathbb{R})$ with $H$ the Heaviside function when $\Re e\:\alpha <0$; otherwise, $x_{+}^{-\alpha-1}$ are not regular distributions and arise from a regularization procedure at the origin as explained in \cite[Section 2.4]{estrada-kanwal}} that they form an entire family of distributions in the parameter $\alpha\in\mathbb{C}$. 
We define $J(s,\alpha)$ via
$$
(s-1)^{\alpha}= J(s,\alpha) +\frac{1}{\Gamma(-\alpha)}\int_{1}^{\infty} e^{-(s-1)x} x^{-1-\alpha}\mathrm{d}x.$$
Differentiating $n$ times with respect to $\alpha$,
 \begin{align*}(s-1)^{\alpha}\log^n (s-1)= \partial^{n}_{\alpha}&J(s,\alpha)
 \\
 & +\int_{1}^{\infty} e^{-(s-1)x}x^{-1-\alpha}\sum_{j=0}^{n}\binom{n}{j} \odv[order=n-j]{}{\alpha}\left(\frac{1}{\Gamma(-\alpha)}\right) (-\log x)^{j} \mathrm{d}x,
\end{align*}
whence we see that $\partial^{n}_{\alpha} J(s,\alpha)$ are Laplace transforms of distributions supported on the compact interval $[0,1]$. The Paley-Wiener-Schwartz theorem \cite{vladimirovbook} then implies that $\partial^{n}_{\alpha} J(s,\alpha)$ are  entire with respect to the variable $s$ as well. Therefore, 
$$\partial^{n}_{\alpha} J(1+it,\alpha)\in C^{\infty}(\mathbb{R})\subset A_{loc}(\mathbb{R}) \qquad \mbox{for each } \alpha\in \mathbb{C}.$$  If $\Re e\:\alpha>0$, we obviously have $x^{-1-\alpha} \log^{j}x \in L^{1}[1,\infty)$ and we infer that the boundary values of $(s-1)^{\alpha}$ on $1+i\R$ belong to the local Wiener algebra $A_{loc}(\mathbb{R})$.  We obtain 
 $D(1+it)\in A_{loc}(0)$ and consequently 
 \[B(1+it) = D(1+it) + F(1+it) \cdot \lim_{\sigma\to0^{+}} (\sigma+it)^{\rho} \in  A_{loc}(0).\] 
 (Here and below we repeatedly apply Lemma \ref{lemma mult}).  As we already pointed out, $B(1)\neq 0$. Lemma \ref{Wiener lemma} hence yields
 $$
 \frac{1}{B(1+it)} \in A_{loc}(0). $$
 We finally notice that  $D'(1+it)\in L_{loc}^{1}(\mathbb{R})\subset PF_{loc}(\mathbb{R})$, because 
 $$
 \lim_{\sigma\to0^{+}} (\sigma+it)^{\alpha}= e^{i\frac{\pi \alpha}{2}} t_{+}^{\alpha} + e^{-i\frac{\pi \alpha}{2}} t_{-}^{\alpha}\in L^{1}_{loc}(\mathbb{R})
 $$
 if $\Re e\:\alpha>-1$ (here $t_{\pm}^{\alpha}=|t|^{\alpha}H(\pm t)$,  where again $H$ stands for the Heaviside function).
  
  To conclude the proof, we decompose $B$ as in \eqref{decompose B} with (see \eqref{decompose B 2})
  \[
  g_{1}(t)=\frac{1}{B(1+it)}(e^{i\frac{\pi \rho}{2}} t_{+}^{\rho} + e^{-i\frac{\pi \rho}{2}} t_{-}^{\rho})\in A_{loc} (0)\cdot A_{loc} (0)= A_{loc} (0)  \qquad \mbox{with }g_1(0)=0,
  \]
  \[
  g_{2}(t)= F'(1+it) \in PM_{loc}(0),
  \]
  and 
  \[g_{3}(t)= \frac{D'(1+it)+\rho F(1+it) (e^{i\frac{\pi (\rho-1)}{2}} t_{+}^{\rho-1} + e^{-i\frac{\pi (\rho-1)}{2}} t_{-}^{\rho-1})}{B(1+it)} \in  L^{1}_{loc}(0) \subset PF_{loc}(0).
  \]
  
\end{proof}

Theorem \ref{eq N intro} can now be derived as a particular instance of Theorem \ref{th CB}. Indeed, extending $E$ as 0 at the left of $e$, the expression \eqref{eq N intro} yields \eqref{th CB} with certain polynomials $P_j$, $\rho_j=r_j+1$, $\rho=r+1>0$,  $c=a\Gamma(\rho)>0$, and $F(s)=F_{1}(s)+F_{2}(s)$ with $F_{2}(s)$ having analytic extension to a complex neighborhood of $s=1$, $F_{1}(s)= s\int_{1}^{\infty} f_{1}(x)e^{-(s-1)x}\mathrm{d}x$, and $f_1(x)=e^{-x}E(e^{x})$. The hypothesis \eqref{Ebound2} translates into $f_1\in L^{1}(\mathbb{R})$ and hence $F(1+it)=F_{2}(1+it) +(it+1) \widehat{f}_{1}(t)\in A_{loc}(0)$. On the other hand, if we set $f_{2}(x)=e^{-x}\int_{e}^{e^{x}} y^{-1}E(y)\log y\: \mathrm{d}y \in L^{\infty}(\mathbb{R})$ (by \eqref{Ebound3}),
\begin{align*}
F'(s)&=F'_{2}(s)+ \int_{1}^{\infty} f_{1}(x)e^{-(s-1)x}\mathrm{d}x-s\int_{e} ^{\infty} x^{-1-s}E(x) \log x \: \mathrm{d}x
\\
&= F'_{2}(s)+ \int_{1}^{\infty} f_{1}(x)e^{-(s-1)x}\mathrm{d}x-s^{2}\int_{1}^{\infty} f_{2}(x)e^{-(s-1)x}\mathrm{d}x.
\end{align*}
We thus see that $F'(1+it)= F_{2}'(1+it)+\widehat{f}_{1}(t)- (1+it)^{2}\widehat{f}_{2}(t) \in PM_{loc}(0)$.

We end this article with two remarks.

\begin{remark}
Let $\zeta(s)$ be convergent on $\Re e \:s>1$. Then, $\zeta(s)$ admits distributional boundary values in a open boundary line segment containing $s=1$ (or equivalently on the whole boundary line $\Re e\:s=1$) if and only if there is some $r>0$ such that
\begin{equation}
\label{eq: r1}
N(x)=O(x \log^ r x).
\end{equation}
In fact, it is obvious that \eqref{eq: r1} implies that $\zeta(1+it)$ is the tempered distribution given by the Fourier transform of $e^{-x}N(e^{x})\in \mathcal{S}'(\R)$. Conversely, if $\zeta$ has distributional boundary values on a boundary neighborhood of $s=1$, there is \cite[pp.~63--66]{hormanderbook} some $r>0$ and a compact neighborhood $I\subset \R$ of 0 such that 
\[
\zeta(1+\sigma+it)\ll \frac{1}{(\sigma+it)^{r}}, \qquad \sigma\to 0^{+},
\]
uniformly for $t\in I$. In particular,
\[
N(x)\leq x \int_{1^{-}} ^{x} \frac{\mathrm{d}N(u)}{u}\leq xe\int_{1^{-}} ^{x}\frac{\mathrm{d}N(u)}{u^{1+1/\log x}}\leq x e \zeta\left(1+\frac{1}{\log x}\right)\ll x\log^{r}x.
\]

\end{remark}

\begin{remark}The Chebyshev bounds under the hypotheses ($a>0$)
\begin{equation}
\label{rCBeq1}
\int_{1}^{\infty}\frac{|N(x)-a|}{x^{2}}{\mathrm{d}x}<\infty
\end{equation}
and 
\begin{equation}
\label{rCBeq2}
N(x)=ax+ O\left(\frac{x}{\log x}\right)
\end{equation}
were first proved in \cite{diamond-zhangbook} (cf. \cite{diamond-zhang13a,VindasChebyshevII}). Hall constructed \cite{hall} an example of a Beurling number system in which the Chebyshev bounds fail but for which $N(x)=ax +O(x/\log^{\alpha} x)$ when $0<\alpha<1$. This was subsequently improved by examples of Kahane \cite{kahane98} and Diamond and Zhang \cite{diamond-zhang13b} (cf. \cite[Chapter 12]{diamond-zhangbook}). Kahane showed that \eqref{rCBeq1} is not strong enough by itself to yield \eqref{eq: CB}, refuting a conjecture by Diamond; while Diamond and Zhang proved that one can always construct a generalized number system without Chebyshev bounds but satisfying \eqref{rCBeq1} and
$$
N(x)=ax +O\left(\frac{x \: \omega(x)  }{\log x}\right)
$$
where $\omega$ is a given arbitrary positive increasing unbounded function.

Whether the implication 
$$\mbox{\eqref{rCBeq2}}\implies \mbox{\eqref{eq: CB}}$$ (unconditionally)  holds true or not remains an important \emph{open question} in Beurling generalized number theory.
\end{remark}

\end{document}